\documentclass[11pt]{amsart}

\usepackage{eucal,graphicx,amssymb,mathrsfs}
\usepackage{bm}
\usepackage[alphabetic]{amsrefs}
\usepackage[all, knot]{xy}

\usepackage{stmaryrd}

\usepackage{setspace}

\theoremstyle{definition}
\newtheorem{ex}{Example}[section]
\newtheorem*{ex*}{Example}
\newtheorem*{defn}{Definition}
\newtheorem*{DMM*}{Dynamical Manin-Mumford Problem}
\newtheorem*{question*}{Question}

\theoremstyle{plain}
\newtheorem{thm}{Theorem}[section]
\newtheorem{prop}[thm]{Proposition}
\newtheorem*{cor*}{Corollary}
\newtheorem{cor}[thm]{Corollary}
\newtheorem{lem}[thm]{Lemma}
\newtheorem*{claim}{Claim}
\newtheorem*{thm*}{Theorem}

\makeatletter
\def\namedlabel#1#2{\begingroup
   \def\@currentlabel{#2}%
   \label{#1}\endgroup
}
\makeatother

\theoremstyle{remark}

\numberwithin{equation}{section}

\DeclareMathOperator{\ch}{ch}

\DeclareMathOperator{\codim}{codim}

\DeclareMathOperator{\ind}{ind}

\def\C{\mathbb{C}}

\def\P{\mathbb{P}}

\def\Z{\mathbb{Z}}
\def\P{\mathbb{P}}

\def\N{\mathbb{N}}

\def\O{{\mathcal{O}}}

\def\U{\mathcal{U}}

\def\SS{\mathbf{S}}

\def\v{\mathbf{v}}

\def\cL{\mathcal{L}}

\def\lbr{\llbracket}
\def\rbr{\rrbracket}

\def\fan{\Delta}
\def\fank{\fan_{\cX_k}}

\def\n0{{\bf n_0}}

\def\bx{{\bf x}}

\def\transp #1{\vphantom{#1}^{\mathrm t}\! {#1}}

\def\cV{\mathcal{V}}
\def\cU{\mathcal{U}}
\def\cX{\mathcal{X}}
\def\cY{\mathcal{Y}}
\def\cC{\mathcal{C}}
\def\cL{\mathcal{L}}

\def\Hess{{\mathcal Hess}}
\def\Flag{{\mathcal Flag}}
\def\Ind{{\mathcal Ind}}
\def\fS{\mathfrak{S}}

\def\ba{{\mathbf a}}

\def\Code{{\mathcal Code}}

\def\l{Z}
\def\bz{{\bf{z}}}

\def\Vflag{V_\bullet}

\def\lr #1{\langle #1\rangle}
\def\origin{\mathbf{0}}

\begin{document}

\title[Hessenberg Varieties Related To The Permutohedral Variety]
{The Geometry and Combinatorics of Some Hessenberg Varieties Related to the Permutohedral Variety}

\author{Jan-Li Lin}
\address{Department of Mathematics and Statistics, Washington University in St. Louis, One Brookings
	Drive, St. Louis, Missouri 63130, U.S.A.}
\email{jan-li@wustl.edu}

\subjclass{}

\keywords{}

\begin{abstract}
We construct a concrete isomorphism from the permutohedral variety to the regular semisimple Hessenberg variety associated to the Hessenberg function $h_+(i)=i+1$, $1\le i\le n-1$. 
In the process of defining the isomorphism, we introduce a sequence of varieties which we call the prepermutohedral varieties. We first determine the toric structure of these varieties and compute the Euler characteristics and the Betti numbers using the theory of toric varieties.
Then, we describe the cohomology of these varieties. We also find a natural way to encode the one-dimensional components of the cohomology using the codes defined by Stembridge~\cite{Ste}.  
Applying the isomorphisms we constructed, we are also able to describe the geometric structure of regular semisimple Hessenberg varieties associated to the Hessenberg function represented by $h_k= (2,3, \cdots, k+1, n,\cdots,n)$, $1\le k\le n-3$. In particular, we are able to write down the cohomology ring of the variety. Finally, we determine the dot representation of the permutation group $\fS_n$ on these varieties. 
\end{abstract}

\maketitle


\section{Introduction}

There are several different ways to describe the permutohedral variety $\cX$ of dimension $n-1$, for instances: 
\begin{enumerate}
	\item It is the toric variety associated to the normal fan of the permutohedron.
	\item It is the graph of the Cremona involution
	\begin{align*}
		J:\P^{n-1}& \dashrightarrow \P^{n-1}\\
	          [z_1 : \cdots : z_{n}]    & \longmapsto  	[z_1^{-1} : \cdots : z_{n}^{-1}] 
	\end{align*}
    \item It is an iterated blowup of $\P^{n-1}$ at all (strict transforms of) coordinate linear subspaces in a certain order, as follows. 
\[
\xymatrix{
	\cX_{n-2}\ar[r]^{\pi_{n-2}} &\cX_{n-3} \ar[r]^{\pi_{n-3}}& \cdots \ar[r]^{\pi_{2}}&
	\cX_{1} \ar[r]^{\pi_{1}\ \ \ \ \ } & \cX_0 = \P^{n-1}
}
\]
Each $\cX_{k+1}$ is the blowup of $\cX_k$ at all the strict transform of the $k$-dimensional coordinate spaces ($0\le k\le n-3$). Further details are in Section~\ref{section:blowups}.
    \item It is the regular semisimple Hessenberg variety $\Hess(\SS,h_+)$ associated to the Hessenberg function $h_+$ defined by $h_+(i)=i+1$ for $1\le i\le n-1$ and $h_+(n)=n$.
\end{enumerate} 

The isomorphisms among the first three descriptions are well-known. On the other hand, to the author's knowledge, the isomorphism between the permutohedral variety and the Hessenberg variety is proved by general theory of toric varieties~\cite[Lemma 10 and Theorem 11]{DMPS}. The first goal of this article is to construct a concrete isomorphism from the iterated blowups on $\P^{n-1}$ to the Hessenberg variety. In the process, we further obtain isomorphisms from $\cX_k$, the $k$-th step in the iterated blowups, to a ``Hessenberg type'' subvariety of a partial flag variety which is denoted by $\Hess^{(k+1)}(\SS,h_+)$ (for details, see Section~3). 

Next, in Section~4, we apply the theory of toric varieties to explore certain geometric properties of the varieties $\cX_k$. In particular, we show that the Euler characteristic of $\cX_k$ is equal to the permutation number $P( n, k+1) = \frac{n!}{(n-k-1)!}$. In fact, there is a basis for the homology $H_*(\cX_k)$ whose elements are in one-to-one correspondence with permutations of $(k+1)$ different numbers chosen from $[n]:=\{1,\cdots,n\}$. As a consequence, the question of finding Betti numbers of $\cX_k$ turns into a counting question of permutations with certain property. More precisely, if we make the following definition: 
\begin{defn}	 
	For a permutation $a_1,\cdots, a_{k+1}$ of $k+1$ different numbers in $[n]$, set $\alpha_0=[n]\setminus \{a_1,\cdots, a_{k+1}\}$, a {\em descent} (resp. {\em ascent}) for the permutation $a_1,\cdots,a_{k+1}$ is either $a_j > a_{j+1}$ (resp. $a_j < a_{j+1}$) for some $j=1,\cdots, k$, or $a > a_{1}$ (resp. $a < a_{1}$) for some $a\in\alpha_0$.
\end{defn} 

Then we can calculate the even Betti numbers of $\cX_k$ as follows (note that the odd Betti numbers of $\cX_{k}$ are all $0$.)

\begin{prop} [Proposition~\ref{prop:Betti}]
The $2i$-th Betti number of $\cX_{k}$ is given by
	\begin{align*}
		\beta_{2i}(\cX_{k})&= \# (\text{permutations of $k+1$ different numbers in $[n]$ with $i$ descents})\\
		&= \# (\text{permutations of $k+1$ different numbers in $[n]$ with $n-1-i$ ascents}).
	\end{align*}
\end{prop}
These Betti numbers are quite natural generalization of the Eulerian numbers. In fact, for $k=n-2$, $\cX_{n-2}$ is the permutohedral variety and it is well known that $\beta_{2i}(\cX_{n-2})=A(n,i+1)$ are the Eulerian numbers. However, the author's cannot find information about them in the literature. It might be interesting to study further properties of these numbers.

In Section 5, we first describe $H^*(\cX_k)$ using the blowup structure (Proposition~\ref{thm:cohomology_preperm}).
There is a system of codes defined by Stembridge in \cite{Ste}. He also proved in the article that the representation of the symmetric group $\fS_n$ on $H^*(\cX)$ is isomorphic to the permutation representation induced by the action on codes \cite[Proposition 4.1]{Ste}. In Proposition~\ref{prop:encoding_cohomology}, we realize the isomorphism by defining a natural one-to-one correspondence between codes of length $n$ and one-dimensional components of $H^*(\cX^{n-1})$. The correspondence can also be restricted to $H^*(\cX^{n-1}_k)$, and is compatible with the $\fS_n$ action. This immediately gives us a way to concretely construct permutation basis for the $\fS_n$ representation on each $H^*(\cX^{n-1}_k)$, $0\le k\le n-2$.

The question of finding a permutation basis for the dot action on the cohomology of Hessenberg varieties has drawn people's attention recently \cite{AHM, CHL, Chow, HPT}  because of it's relation with the Stanley-Stembridge conjecture. The relation was observed by Shareshian and Wahcs~\cite{SW}. They also announced an important conjecture that was later proved by Brosnan and Chow~\cite{BC}, and independently by Guay-Paquet~\cite{GP}. In the case of the permutohedral variety, the Stanley-Stembridge conjecture is known to be true, and a permutation basis is also known. The known basis was first conjectured by Chow \cite{Chow} and then proved by Cho, Hong, and Lee \cite{CHL}. It is based on the theory on equivariant cohomology, then pass it on to the usual cohomology. The basis constructed in this paper is based on the geometric structure of $\Hess^{(k)}(\SS,h_+)$ and the combinatorics of the codes, and is for the usual cohomology. It might be interesting to compare these two kinds of basis.

In Section 6, we use the isomorphisms at each step of the blowups to investigate semisimple Hessenberg varieties $\Hess(\SS,h_k)$ associated to the Hessenberg function 
$h_k=(2,3,\cdots,k+1,n,\cdots, n)$, $1\le k\le n-3$.
We observe that $\Hess(\SS,h_k)$ has a fiber bundle structure over $\cX_k$ with fibers isomorphic to the flag variety $\Flag(\C^{n-k-1})$. We use this structure to obtain a description of the cohomology of $\Hess(\SS,h_k)$. 

\begin{prop} [Proposition \ref{prop:cohomology_Y}]
	The cohomology ring $H^*(\Hess(\SS,h_k))$ is generated over $H^*(\cX_{k})$ as
	\[
	H^*(\Hess(\SS,h_k))\cong H^*(\cX_{k})[X_{k+2},\cdots,X_n]/
	(e_1(X_{k+2},\cdots,X_n),\cdots,e_{n-k-1}(X_{k+2},\cdots,X_n)),
	\]
where $e_j$ is the $j$-th elementary symmetric polynomials.
In addition, for the images $x_i$'s of the $X_i$'s, the classes $x_{k+2}^{i_{k+2}}\cdots x_n^{i_n}$, with exponents $0\le i_{j}\le n-j$, form a basis for $H^*(\cY)$ over $H^*(\cX_{k})$. 
\end{prop}

Finally, we determine the dot representation on the cohomologies of $\Hess(\SS,h_k)$.

 \begin{prop} [Proposition \ref{prop:dot_on_Y}]
 	The dot representation on $\Hess(\SS,h_k)$ is isomorphic to the representation on 
 	\[
 	H^*(\cX_{k})[X_{k+2},\cdots,X_n]/
 	(e_1(X_{k+2},\cdots,X_n),\cdots,e_{n-k}(X_{k+2},\cdots,X_n))
 	\]
 	which acts on  $H^*(\cX_{k})$ as described in Section \ref{sec:dot_on_preperm}, and acts trivially on the $H^*(\cX_{k})$-basis $x_{k+2}^{i_{k+2}}\cdots x_n^{i_n}$, $0\le i_{j}\le n-j$.
 \end{prop}

Our proof is based on the characteristic series of the representation, and the conclusion is up to an isomorphism. It would be interesting to compute the $\fS_n$ action on the basis elements; presumably first on the equivariant cohomology then pass to the usual cohomology, as was done in~\cite{Tymo2}.


\section{The permutohedral variety as iterated blowups of $\P^{n-1}$}
\label{section:blowups}

One can obtain the permutohedral variety by performing a sequence of blowups on $\P^{n-1}$, as follows.

\begin{enumerate}
	\item First, we blowup the $n$ points $\l_1 = [1:0:\cdots:0]$, $\l_2 = [0:1:0\cdots:0]$, $\cdots$, $\l_n = [0:\cdots:0:n]$. 
	We denote the resulting variety and the projection map by $\pi_1:\cX_1\to\P^{n-1}:=\cX_0$. We also have the exceptional divisors $E_i=\pi_1^{-1}(\l_i)\subset\cX_1$.
	
	\item Next, we blowup the strict transforms (in $\cX_1$) of all the lines $\l_{\{i,j\}}\subset \P^{n-1}$ connecting $\l_i$ and $\l_j$ for $1\le i < j \le n$. This gives us the second level space $\pi_2:\cX_2\to\cX_1$ and the exceptional divisors $E_{\{i,j\}}\subset \cX_2$. Notice that, although the lines $\l_{\{i,j\}}$ and $\l_{\{i,k\}}$ intersect at $\l_i$ in $\P^{n-1}$, the blowups in step 1 would separate their strict transforms. Therefore, the resulting space $\cX_2$ is independent of the order of blowups.
	
	\item We repeat the above process until we reach codimension $2$. More precisely, in the $k$-th step ($1\le k\le n-2$) we do the following. For $\alpha\subset [n]$ a subset of $k$ elements, let $\l_\alpha$ denote the linear subvariety of $\P^{n-1}$ generated by $\{\l_i|i\in\alpha\}$, and $\overline{\l}_\alpha$ the strict transform of $\l_\alpha$ in $\cX_{k-1}$. The blowups in the previous steps have the effect of blowing up all coordinate linear subspaces on $\l_\alpha$, thus $\overline{\l}_\alpha$ is a permutohedral variety of dimension $k-1$. We blow up all the $\overline{\l}_\alpha$ in this step. The $\overline{\l}_\alpha$'s intersect with each other along coordinate subspaces of lower dimensions, hence are separated by previous blowups. Thus, the order to perform blowups in this step does not matter. This produces the space $\cX_k$, the map $\pi_k:\cX_k\to\cX_{k-1}$, as well as the exceptional divisors $E_\alpha\subset\cX_k$
\end{enumerate}

The end result is a sequence of spaces and projection maps:
\[
\xymatrix{
	\cX_{n-2}\ar[r]^{\pi_{n-2}} &\cX_{n-3} \ar[r]^{\pi_{n-3}}& \cdots \ar[r]^{\pi_{2}}&
	\cX_{1} \ar[r]^{\pi_{1}\ \ \ \ \ } & \cX_0 = \P^{n-1}
}
\]
The variety $\cX=\cX_{n-2}$ is the permutohedral variety. 

\begin{defn}
	We call the varieties $\cX_k$ ($0\le k\le n-2$) the {\em prepermutohedral variety} of order $k$. 
\end{defn}

\section{Permutohedral varieties as Hessenberg varieties}

We consider regular semisimple Hessenberg varieties of type A. To set the notations, let $\SS$ be an $n\times n$ complex diagonal matrix with different diagonal entries $s_1,\cdots,s_n$ and $h:[n]\to[n]$ be a Hessenberg function, i.e., $h(i)\ge h(j)\text{ for all $n\ge i>j\ge 1$}$ ($h$ is non-decreasing) and $h(i)\ge i$ for $i=1,\cdots,n$. 
\[
\Hess(\SS,h)=\left\{ \Vflag = (\{0\}=V_0\subset V_1\subset \cdots\subset V_n=\C^n)\  |\ \SS(V_i)\subset V_{h(i)} \text{ for $i=1,\cdots,n$}\right\}.
\] 
We also use the list notation $h=(h(1),\cdots,h(n))$ to denote a Hessenberg function. In this section, we focus on the specific Hessenberg function $h_+=(2,3,\cdots, n-1,n,n)$, i.e. $h_+(i)=\min(i+1,n)$ for $1\le i\le n$. 
We start with an observation in linear algebra. 
\begin{lem}
\label{lem:lin_rel}
Let $\SS$ be an $n\times n$ complex diagonal matrix with different diagonal entries $s_1,\cdots,s_n$ and $\v=\transp{(z_1, \cdots,z_n)}\in \C^n$ be a (column) vector. If at least $k$ of the $z_i$'s are nonzero (i.e. there are no more than $n-k$ of the $z_i=0$), then the vectors
\[
\v,\ \SS\v,\ \cdots,\ \SS^{k-1}\v
\]
are linearly independent. If exactly $k$ of the $z_i$'s are nonzero, then the vectors
\[
\v,\ \SS\v,\ \cdots,\ \SS^{k}\v
\]
are linearly dependent.
\end{lem}

\begin{proof}
Assuming $z_i\ne 0$, then a linear relation 
\[
a_0 \v+ a_1 \SS\v + \cdots + a_{k-1} \SS^{k-1}\v =0
\]
implies the relation (on the first coordinate)
\[
a_0 + a_1 s_i + \cdots + a_{k-1} s_i^{k-1} = 0.
\]
That is, $s_i$ is a root of the polynomial $\sum_{i=0}^{k-1} a_i z^i$. The first part of the lemma is then a consequence of the fact that a polynomial equation of degree $k-1$ cannot have $k$ or more distinct roots.

For the second part of the lemma, assume that $z_1,\cdots,z_k$ are the non-zero $z_i$'s. Then the coefficients of the polynomial $(z-s_1)\cdots(z-s_k)$ give a non-trivial linear relation among $\v,\ \SS\v,\ \cdots,\ \SS^{k}\v$.
\end{proof}

In particular, if $\origin=(0,\cdots,0)$ denotes the origin, and $\v\in\C^n\setminus\{\origin\}$, 
then $\v$ gives rise to a point $[\v]\in \P^{n-1}$. We further denote
\[
\Ind := \bigcup_{1\le i<j\le n}(z_i=z_j=0).
\]
The set $\Ind$ is the indeterminate set of the Cremona involution $J$ defined in the introduction.
If $\v\in \P^{n-1}\setminus \Ind$, then at least $n-1$ of the $x_i$ are nonzero. Thus, by the lemma, the vectors
\[
\v,\ \SS\v,\ \cdots,\ \SS^{n-1}\v
\]
are linearly independent, and the following is a well-defined flag (i.e. the dimension of the vector spaces are correct).
\[
\Vflag=\left(\lr{\origin}\subset\lr{\v}\subset\lr{\v, \SS\v}\subset\cdots\subset\lr{\v,\SS\v,\cdots\SS^{n-2}\v}\subset\C^n\right).
\]
Moreover, it is obvious that $\Vflag\in \Hess(\SS,h_+)$. Conversely, suppose that $\Vflag = \left(\{0\}\subset V_1\subset \cdots\subset V_{n-1}\subset\C^n\right)$ is a flag in $\Hess(\SS,h_+)$, and the one-dimensional space $V_1$, as an element of $\P^{n-1}$, satisfies $V_1\not\in \Ind$, then $\Vflag$ must be in the form 
\[
\Vflag=\left(\lr{\origin}\subset\lr{\v}\subset\lr{\v, \SS\v}\subset\cdots\subset\lr{\v,\SS\v,\cdots\SS^{n-2}\v}\subset\C^n\right)
\]
for any nonzero $\v\in V_1$. This defines an isomorphism 
\begin{align*}
	 \P^{n-1}\setminus\Ind 
	       & \longrightarrow \cU \subset \Hess(\SS,h_+)\\
	\v     & \longmapsto  
	         \left(\lr{\origin}\subset\lr{\v}\subset\lr{\v, \SS\v}\subset\cdots\subset\lr{\v,\SS\v,\cdots\SS^{n-2}\v}\subset\C^n\right),
\end{align*}
where $\cU$ is the open subset of $\Hess(\SS,h_+)$ consists of all flags $\Vflag = \left(\{0\}\subset V_1\subset \cdots\subset V_{n-1}\subset\C^n\right)$ such that $V_1\not\in \Ind$.
We will extend this isomorphism to isomorphisms between blowups of $\P^{n-1}$ and Hessenberg type varieties defined in the next paragraph.

In order to do so, we introduce a type of partial flag variety. For $0\le k\le n-2$, define 
\[
\Flag^{(k+1)}(\C^n):=\left\{\Vflag=(V_0:=\lr{\origin}\subset V_1\subset \cdots \subset V_{k+1})\ |\ \dim_\C(V_i)=i\text{ for $i=1,\cdots,k+1$}\right\}. 
\]
We also define the corresponding Hessenberg type variety
\[
\Hess^{(k+1)}(\SS,h_+):=\left\{\Vflag\in \Flag^{(k+1)}(\C^n)\ |\ \SS V_i\subset V_{i+1}\text{ for $i=0,\cdots,k$}\right\}.
\]
Notice that $\Flag^{(n-1)}(\C^n)=\Flag(\C^n)$ and $\Hess^{(n-1)}(\SS,h_+)=\Hess(\SS,h_+)$. 
The main goal of this section is to show the following. 
\begin{prop}
There is a natural isomorphism 
\[ 
\cX_{k}\stackrel{\cong}{\longrightarrow} \Hess^{(k+1)}(\SS,h_+)
\]
for $k=0,\cdots,n-2$, where $\cX_{k}$ is the prepermutohedral variety.
\end{prop}

\begin{proof}
It is clear that $\cX_0:=\P^{n-1}\cong \Hess^{(1)}(\SS,h_+)$. For $\Vflag= (V_0:=\lr{\origin}\subset V_1\subset V_2)\in \Hess^{(2)}(\SS,h_+)$, suppose that $\origin\ne\v=(z_1,\cdots,z_n)\in V_1$.  By lemma~\ref{lem:lin_rel}, if at least two of the $x_i$'s are non-zero, then $\v$ and $\SS\v$ are linearly independent, and $V_2=\lr{\v,\SS\v}$ is determined. 
If only one of the $z_i\ne 0$, then $V_1=\l_i$ (see Section~\ref{section:blowups} for the definition of $\l_i$) for some $i=1,\cdots,n$ and thus $\SS V_1=V_1$. To determine $V_2$, we need an extra piece of information, which is the direction that give us the second dimension of $V_2$. This can be specified as a point in $\P(\C^n/V_1)$, which is canonically isomorphic to the exceptional divisor $E_i$ when we blowup $\P^{n-1}$ at $\l_i$. Therefore, after we blowup all the $\l_i$'s, then for those $V_1$ such that $\SS V_1=V_1$, we also know what $V_2$ is. This gives the isomorphism $\cX_1\to\Hess^{(2)}(\SS,h_+)$. 

We continue the blowup process inductively as follows. First, we introduce the ``forgetful'' morphism
\[
f^{(k)}: \Hess^{(k+1)}(\SS,h_+)\to\Hess^{(k)}(\SS,h_+) \text{ for $k=1,\cdots, n-3$,}
\] 
which sends a flag of $k+1$ vector spaces to the first $k$ of them, and ``forgets'' the last vector space. 
\begin{claim}
The morphism $f^{(k)}$ is a birational map.	
\end{claim}

\begin{proof}
Given $\Vflag= (V_0:=\lr{\origin}\subset V_1\subset\cdots\subset V_{k})\in \Hess^{(k)}(\SS,h_+)$, since $\SS V_{k-1}\subset V_{k}$, we know from linear algebra that $\dim_\C(\lr{V_{k}\cup\SS V_{k}})=k\text{ or } k+1$. Moreover, $\dim_\C(\lr{V_{k}\cup\SS V_{k}})=k+1$ is the generic situation in $\Hess^{(k)}(\SS,h_+)$ and 
the set of all $\Vflag$ such that 
$\dim_\C(\lr{V_{k-1}\cup\SS V_{k}})=k$ is a closed subset of $\Hess^{(k)}(\SS,h_+)$. If $\dim_\C(\lr{V_{k}\cup\SS V_{k}})=k+1$, then sending 
\[
(V_0\subset V_1\subset\cdots\subset V_{k})\longmapsto (V_0\subset V_1\subset\cdots\subset V_{k}\subset \lr{V_{k}\cup\SS V_{k}})
\]
gives the inverse of $f^{(k)}$ for a generic flag in $\Hess^{(k)}(\SS,h_+)$. 	
\end{proof}

If $\dim_\C(\lr{V_{k}\cup\SS V_{k}})=k$, then $\SS V_{k}=V_{k}$, and a reflection of Lemma~\ref{lem:lin_rel} tells us that, via the isomorphism $\cX_{k}\cong \Hess^{(k)}(\SS,h_+)$, $\Vflag$ lies in the strict transform of $\l_\alpha$, denoted by $\overline{\l}_\alpha$, for some $\alpha\subset [n]$ with $k$ elements. 

If we blowup $\cX_{k-1}$ along $\overline{\l}_\alpha$, then the exceptional divisor $E_\alpha$ is canonically identified with the projective normal bundle $\P(N_{\l_{\alpha}\subset X_{k-1}})$. 
A point on $E_\alpha$ carries the information of $\Vflag$, together with a (projective) normal direction of $\l_\alpha$ in $\C^n$, i.e. an element in $\P(\C^n/\l_\alpha)$. This assigns a unique flag in $\Hess^{(k+1)}(\SS,h_+)$. More precisely, if $\v\in\C^n/\l_\alpha$ represents the direction in $\P(\C^n/\l_\alpha)$, then one sets $V_{k+1}=\lr{V_k,\v}$. Once we blowup $\cX_{k-1}$ along the strict transforms of all $\l_{\alpha}$ for $\alpha\subset [n]$ with $k$ elements, we obtain the isomorphism $\cX_{k}\cong \Hess^{(k+1)}(\SS,h_+)$.
\end{proof}

\begin{ex}
	Suppose we have $\v_1=(1,1,0,0,0)\in\C^5$, then $V_1=\lr{\v_1}\subset V_2=\lr{\v_1,\SS\v_1}$ form the first two spaces in the flag but $\SS V_2=V_2$. Thus we need to blowup $\l_{\{1,2\}}$. Notice that $V_2=\l_{\{1,2\}}\subset\C^5$ and elements in $E_{\{1,2\}}$ over $\v_1$ are in one-to-one correspondence to projectivized normal vectors for the embedding $\C^2\hookrightarrow\C^5$ at $\v_1$, i.e., $\P(\C^5/\C^2)$. Suppose we pick $\v_2=(1,1,1,0,0)$ as a representative for an element in $\P(\C^5/\C^2)$, then that gives us $V_3=\lr{V_2,\v_2}$ (one can check that $V_3$ is independent of the choice of $\v_2$), but then we would have $\SS V_3=V_3$ again. This means that the element represented by $\v_2$ in $E_{\{1,2\}}$ is in the strict transform of the set $\l_{\{1,2,3\}}$. When we blowup $\l_{\{1,2,3\}}$, elements of $E_{\{1,2,3\}}$ over $\v_2$ will then corresponds to $\P(\C^5/\C^3)$. Picking a representative, say $\v_3=(1,1,1,1,1)$, we will have $V_4=\lr{V_3,\v_3}$ and $V_5=\C^5$. This gives a flag 
	$\Vflag=\left(\lr{\origin}\subset V_1\subset \cdots \subset V_5\right)\in  \Hess(\SS,h_+)$. 
\end{ex}

\section{Prepermutohedral varieties as toric varieties}

From the isomorphism constructed in the previous section, we can discover how the torus $(\C^*)^{n-1}$ sits inside $\Hess^{(k+1)}(\SS,h_+)$, as follows. The point $\bz=(z_1,\cdots,z_{n-1})\in (\C^*)^{n-1}$ corresponds to the flag
\[
\Vflag:=\left(\lr{\origin}\subset\lr{\v}\subset\lr{\v, \SS\v}\subset\cdots\subset\lr{\v,\SS\v,\cdots\SS^{k}\v}\subset\C^n\right)
\]
where $\v=(1,z_1,\cdots,z_{n-1})$. We denote this correspondence as an injective map
$\phi_{k+1}:(\C^*)^{n-1}\hookrightarrow \Hess^{(k+1)}(\SS,h_+)$ by $\phi(\bz)=\Vflag$.

  One can also observe the algebraic group structure of $(\C^*)^{n-1}\subset\Hess^{(k+1)}(\SS,h_+)$.
Given $(z'_1,\cdots,z'_{n-1})\in (\C^*)^{n-1}$ and $\v'=(1,z'_1,\cdots,z'_{n-1})$, the product of 
$(z_1,\cdots,z_{n-1})$ and $(z'_1,\cdots,z'_{n-1})$ corresponds to the flag
\[
\left(\lr{\origin}\subset\lr{\v\v'}\subset\lr{\v\v', \SS(\v\v')}\subset\cdots\subset\lr{\v\v',\SS(\v\v'),\cdots,\SS^{k}(\v\v')}\subset\C^n\right),
\]
where $\v\v'$ is the coordinate-wise product of $\v$ and $\v'$. That is, $\phi(\bz)\phi(\bz'):=\phi(\bz\bz')$.
Moreover, one can discover the action of $(\C^*)^{n-1}$ on $\Hess^{(k+1)}(\SS,h_+)$ similarly.

\subsection{Fan structure for the prepermutohedral varieties}
Given that $\cX_k\cong \Hess^{(k+1)}(\SS,h_+)$ is a toric variety, we would like to know the structure of the corresponding fan, and the geometry properties we can conclude from the fan structure. The standard reference for this part is \cite{Fulton, CLS}. 

First, recall the structure of the fan $\fan_{\P^{n-1}}$ corresponding to $\P^{n-1}$. Let $e_1,\cdots,e_{n-1}$ be the standard basis of $\Z^{n-1}$ and $e_n=-(e_1+\cdots+e_{n-1})$. The cones in $\fan_{\P^{n-1}}$ are generated by proper subsets of $\{e_1,\cdots,e_n\}$. More precisely, for any proper subset $\alpha\subset [n]$, we denote $\sigma_\alpha$ to be the cone generated by $\{e_i|i\in\alpha\}$, i.e., $\sigma_\alpha=\{\sum_{i\in\alpha}a_i e_i|a_i\ge 0\}$. Then
\[
\fan_{\P^{n-1}} = \left\{ \sigma_\alpha \ |\ \alpha\subset [n]\right\},
\]
with the convention that $\sigma_\emptyset=\{\origin\}$.

As we showed earlier, the prepermutohedral variety $\cX_k$ ($1\le k \le n-2$) is obtained from $\P^{n-1}$ by blowing up the torus invariant subvarieties (in the order of dimensions) up to dimensions $k-1$. For the fan structure, this means we perform the star subdivisions (see \cite[Definition 3.3.17]{CLS}) on all cones of codimensions $0, 1,$ and so on, up to cones of codimension $k-1$. 
We will describe the fan $\fank$ corresponds to the prepermutohedral variety $\cX_k$ after setting up some notations. 

For a non-empty proper subset $\alpha\subset[n]$, we define the vector $e_\alpha:=\sum_{i\in\alpha} e_i$. The rays (i.e. one-dimensional cones) of $\fank$ are generated by $e_\alpha$ for non-empty proper subset $\alpha\subset[n]$ such that $|\alpha|\ge n-k$ or $|\alpha|=1$. To simplify notation, we write $e_i$ instead of $e_{\{i\}}$ for $i\in[n]$.

A chain of subsets of $[n]$ is a sequence of strict inclusions  
\[
\cC=(\alpha_0\subset \cdots\subset \alpha_p)
\]
of proper subsets of $[n]$. It is also allowed that $p=0$, i.e., $\cC=(\alpha_0)$ is a chain consisting of only one set. 
We will also use an alternative list notation to denote a chain. We use special bracket symbols 
$\lbr\ \rbr$ to enclose the list: first list the numbers in $\alpha_0$, then list the numbers in $\alpha_1\setminus \alpha_0$ (separated by a bar), and so on, all the way to $[n]\setminus \alpha_p$. For example, suppose $n=9$ and we have
\[
\cC=(\{1,4\}\subset \{1,2,3,4\} \subset \{1,2,3,4,6,7,9\}),
\]
then in the list notation, $\cC=\lbr 1,4 | 2,3 | 6,7,9 |5,8\rbr$.

Cones in $\fank$ are in one-to-one correspondence with chains $\cC$
such that $|\alpha_0|\le n-k-1$ and $|\alpha_j|\ge n-k$ for $j>0$. The correspondence is given as follows. For each such chain $\cC$, we associate it with the cone $\sigma_\cC$ generated by the vectors $e_i$, $i\in \alpha_0$ and $e_{\alpha_j}$, $j\ge 1$. The dimension of the cone $\sigma_\cC$ is equal to 
$|\alpha_0|+p$.

In particular, the top dimensional cones of $\fank$ are in one-to-one correspondence with chains of the form
\[
\cC=(\alpha_0\subset \alpha_{n-k}\subset \cdots\subset\alpha_{n-1})
\]
such that $|\alpha_0|=n-k-1$ and $|\alpha_j|=j$ for all $j=n-k,\cdots, n-1$. (For $k=0$, this means $\cC=(\alpha_0)$ with $|\alpha_0|=n-1$). For $k\ge 2$, such chains are determined by the sequence of numbers
\begin{align*}
a_n&:=[n]\setminus\alpha_{n-1},\\
a_j&:= \alpha_{j}\setminus\alpha_{j-1},\ \ n-k+1\le j\le n-1\\ 
a_{n-k}&:=\alpha_{n-k}\setminus\alpha_0.
\end{align*}

For $k=0$, such chains are determined by $a_n = [n]\setminus\alpha_0$; and for $k=1$ the chains are determined by $a_n, a_{n-1} = \alpha_{n-1}\setminus\alpha_0$. In the list notation, by a slightly abuse of notation, we would write 
$\cC=\lbr \alpha_0 | a_{n-k} | \cdots | a_n\rbr$.

Notice that, since $\alpha_0$ can be written as $\alpha_0=[n]\setminus \{a_j|n-k\le j\le n\}$, the information on the list of numbers $a_{n-k},\cdots, a_n$ is sufficient to determine the chain $\cC$.
Therefore, the top dimensional cones in $\fank$ are in one-to-one correspondence with permutations of $k+1$ different numbers from $[n]$. The number of top dimensional cones in $\fank$ is then given by the permutation number 
\[
P(n,k+1)=\frac{n!}{(n-k-1)!}.
\]
A result in toric varieties \cite[Section 3.2]{Fulton} states that the Euler characteristic of a toric variety is equal to the number of top dimensional cones in its associated fan. Therefore, we conclude the following.

\begin{prop}
The Euler characteristic of the variety $\cX_{k}$ is given by
\[
\chi(\cX_{k})=P(n,k+1).
\]
\end{prop}

We remark here that one can also calculate the Euler characteristics using the structure of the cohomology groups in the next section.

\subsection{Betti numbers of the prepermutohedral varieties}

In order to calculate the Betti numbers for $\cX_k$, we need to study more on the cone structure of the fan associated to $\cX_k$.

The intersection of two cones is given by the following operation on chains: Given 
$\cC=(\alpha_0\subset \cdots\subset \alpha_p)$ and 
$\cC'=(\alpha'_0\subset \cdots\subset \alpha'_{p'})$, define 
$\cC\cap\cC'$ to be the following.
\begin{itemize}
	\item The first set in the chain is $\alpha_0\cap\alpha'_0$.
	\item The rest of the chain is the greatest common subchain of the chains of sets $\alpha_1\subset\cdots\subset\alpha_p$ and $\alpha'_1\subset\cdots\subset\alpha'_{p'}$
\end{itemize}
 For instance, suppose $n=10$, 
$\cC=\lbr 1,4,10 | 2,3 | 6,7|9 |5,8\rbr$, and 
$\cC'=\lbr 1,4,6 | 7,2,3 | 5,9|8\rbr$, then 
$\cC\cap\cC' = \lbr 1,4 | 2,3,6,7 | 5,8, 9 \rbr$.
With this definition on the intersection of chains, we then have $\sigma_{\cC}\cap\sigma_{\cC'}=\sigma_{\cC\cap\cC'}$.

Also, with the same notations as above, for the chains $\cC, \cC'$, the corresponding cones $\sigma_{\cC}\subset \sigma_{\cC'}$ if $\alpha_0\subset\alpha'_0$ and for all $j>0$, $\alpha_j=\alpha'_{j'}$ for some $j'$.

We would like to assign a total order `$>$' on the top dimensional cones of $\fank$ with some desired property. Equivalently, we will define the order on the set of chains. Set $\tau_\cC$ to be the intersection of $\sigma_\cC$ with all $\sigma_{\cC'}$ that comes after  $\sigma_\cC$ (i.e. $\cC' > \cC$) and such that $\dim(\sigma_{\cC}\cap\sigma_{\cC'})=n-2$. The desired property is that
\begin{equation*} \label{condition_cones}
	\text{If $\tau_\cC \subset \sigma_{\cC'}$, then $\cC' \ge\cC$ } \tag{$*$}
\end{equation*} 

This is also the condition ($*$) in \cite[Section 5.2]{Fulton}. It implies that the classes $[V(\tau_\cC)]$ form a basis for $H_*(X;\Z)$ (see \cite[Theorem on p.102]{Fulton}). 
We claim that the reversed lexicographic order in the list notation for chains will satisfy (\ref{condition_cones}). 

To define the order, suppose $\cC=\lbr \alpha_0 | a_{n-k}|\cdots|a_n\rbr$ and $\cC'=\lbr\alpha'_0| a'_{n-k}|\cdots| a'_n\rbr$. We define $\cC<\cC'$ if for the greatest $j$ such that $a_j\ne a'_j$, we have $a_j<a'_j$. 

Under the above notations and definitions, there are two cases for which we can have $\cC'>\cC$ and $\dim(\sigma_{\cC}\cap\sigma_{\cC'})=n-2$:

\begin{itemize}
	\item[Case 1:] $\alpha_0=\alpha'_0$. In this case, we have $a'_{j+1}=a_j> a'_j = a_{j+1}$ for some $n-k\le j\le n-1$ and $a_i=a'_i$ for all other $n-k\le i\le n$, $i\ne j, j+1$.
	\item[Case 2:] $\alpha_0\ne\alpha'_0$. In this case, we have $a'_{n-k}>a_{n-k}$ and $a_i=a_i'$ for all $i>n-k$.
\end{itemize}

We need the following definition to describe the intersection $\tau_\cC$.

\begin{defn}
	Given a chain $\cC=\lbr \alpha_0 | a_{n-k}|\cdots|a_n\rbr$ in the list notation, a {\em descent} (resp. {\em ascent}) in $\cC$ is either $a_j > a_{j+1}$ (resp. $a_j < a_{j+1}$) for some $j=n-k,\cdots, n-1$, or $a > a_{n-k}$ (resp. $a < a_{n-k}$) for some $a\in\alpha_0$.
\end{defn}

We can now describe $\tau_\cC$. It is obtained from $\sigma_\cC$ through the following process.
\begin{itemize}
	\item Whenever there is a descent $a_j > a_{j+1}$, one removes the set $\alpha_j$ in the chain $\cC$.
	\item Whenever there is a descent $a > a_{n-k}$, one removes the number $a$ from the set $\alpha_0$.
\end{itemize}

For instance, suppose $n=9, k=4$ and $\cC=\lbr 1,2, 5, 8    | 4 | 3 | 6| 9| 7 \rbr$, then there are 
$4$ descents in $\cC$: $5>4, 8>4, 4>3$, and  $9>7$. The cone of intersections 
$\tau_\cC= \lbr 1,2| 5, 8 , 4 , 3 | 6| 9, 7 \rbr$, and $\dim(\tau_\cC)=5$.

One observes that every time there is an descent in $\cC$, the dimension of $\tau_\cC$ would go down by $1$. Therefore, the dimension of $\tau_\cC$ is given by
\[
\dim(\tau_\cC)= n-1 - \#(\text{descents in $\cC$)}= \#(\text{ascents in $\cC$}).
\] 
The complex dimension of corresponding orbit closure $V[\tau_\cC]$ will then be
\begin{align*}
\dim_\C(V[\tau_\cC]) &= n-1-\dim(\tau_\cC)\\
&= \#(\text{descents in $\cC$)}= n-1-\#(\text{ascents in $\cC$}).
\end{align*}
Notice that $\cC$ is the chain of a top dimensional cone with the property that for each block (numbers between two bars) in the chain corresponds to $\tau_\cC$, numbers in the block are descending in $\cC$. Thus, it is the smallest in the reverse lexicographic order among the longest chains that refine the chain for $\tau_\cC$, and we can conclude that  $\tau_\cC \subset \sigma_{\cC'}$ implies $\cC' \ge\cC$. Therefore, by \cite[Theorem on p.102]{Fulton}, the homology classes of orbit closures $V[\tau_\cC]$, as $\cC$ runs through all top dimensional cones, will form a basis of $H_*(X;\Z)$. This immediately implies the following.

\begin{prop}
	\label{prop:Betti}
The $2i$-th Betti number of $\cX_{k}$ is given by
\begin{align*}
	\beta_{2i}(\cX_{k})&= \# (\text{permutations of $k+1$ different numbers in $[n]$ with $i$ descents})\\
	&= \# (\text{permutations of $k+1$ different numbers in $[n]$ with $n-1-i$ ascents}).
\end{align*}
\end{prop}

We remark here that it is possible to find some recursive relations among these Betti numbers using the recursive structure of the cohomology groups discuss in the next section (\ref{eq:cohomology_preperm}).

\section{The dot action on cohomology groups}
\label{sec:dot_on_preperm}

We investigate the dot action on the cohomology groups of the pre-permutohedral varieties. We work on the usual cohomology groups.
In this section and the next, we sometimes need to specify the dimension of a prepermutohedral variety. In such situation, we will use $\cX^{n-1}_k$ to denote a prepermutohedral variety of dimension $n-1$ and of order $k$.

\subsection{The cohomology groups of prepermutohedral varieties} First, we recall the standard result on the cohomology of blowup spaces. For more details, see \cite[\S 7.3.3]{Voisin}.
	Let $X$ be a K\"{a}hler manifold of dimension $n$, $Z\subset X$ be a submanifold, and $\tilde{X}_Z\stackrel{\tau}{\longrightarrow}X$ be the blowup of $X$ along $Z$. Then $\tilde{X}_Z$ is still a K\"{a}hler manifold. Let $E=\tau^{-1}(Z)$ be the exceptional divisor, then $E=\P(N_{Z\subset X})$ is the projective bundle of the normal bundle of $Z$ in $X$. Thus $E$ is of rank $r-1$, where $r=\codim(Z)$, and $E\stackrel{j}{\hookrightarrow}X$ is a hypersurface.

\begin{thm}
	\cite[Theorem 7.31]{Voisin}
	\label{thm:blowup}
 Let $h=c_1(\O_E(1))\in H^2(E;\Z)$, we have isomorphism
	\[
	\xymatrix{
		H^p(X;\Z)\oplus\left(\bigoplus_{i=0}^{r-2} H^{p-2i-2}(Z;\Z)\right)
		\ar[rrrrr]^{\tau^*+\sum_i j_* \circ (\cup h^i) \circ \tau|_E^*} &&&&& H^p(\tilde{X}_Z;\Z).
	}
	\]
The map on the second component is decomposed as follows.
\[
	\xymatrix{
	H^{p-2i-2}(Z;\Z)
	\ar[r]^{\tau|_E^*} & H^{p-2i-2}(E;\Z)
	\ar[r]^{\cup h^i} & H^{p-2}(E;\Z)
	\ar[r]^{j_*} & H^p(\tilde{X}_Z;\Z).
}
\]
Here $j_*$ is the Gysin morphism which is defined as the Poincare dual of the map
\[
j_*: H_{2n-p}(E;\Z) \to H_{2n-p}(\tilde{X}_Z;\Z).
\]

\end{thm}

Some remarks:
\begin{itemize}
	\item The theorem is also valid for cohomology groups with $\C$ coefficients. In this paper, we always use complex coefficients, thus we will write $H^p(X)$ for $H^p(X;\C)$ from now on.
	\item We denote $H^*(X)=\oplus_p H^p(X)$ as a graded complex vector space. Moreover, define $H^p_q(X):=H^{p+q}(X)$, i.e. the degree $p$ part of $H_q^*(X)$ is equal to the degree $p+q$ part of $H^*(X)$. The theorem can be written in term of the graded $\C$ vector spaces as
\[
H^*(\tilde{X}_Z)\cong H^*(X)\oplus\left(\bigoplus_{i=1}^{r-1} H^*_{-2i}(Z)\right).
\]
\end{itemize}

In the process of obtaining $\cX_k$, we blowup strict transforms of coordinate linear subvarieties $\l_\alpha$. The strict transforms $\overline{\l}_\alpha$ are themselves permutohedral varieties of lower dimension. We will use $\cX^d$ to denote the permutohedral variety of dimension $d$. A more compact form of the following description is in \cite{Procesi}.

\begin{thm}
\label{thm:cohomology_preperm}
The cohomology, as a graded complex vector space, of the pre-permutohedral variety $\cX_k^{n-1}$, $1\le k \le n-2$, is given by
\begin{equation}
\label{eq:cohomology_preperm}
H^*(\cX_k^{n-1}) \cong H^*(\P^{n-1})\oplus\left(
\bigoplus_{j=1}^k \bigoplus_{\stackrel{\alpha\subset[n],}{|\alpha|=j}} \bigoplus_{i=1}^{n-j-1} H^*_{-2i}(\cX^{j-1})\right).
\end{equation}
\end{thm}

\begin{proof}
	We apply Theorem~\ref{thm:blowup} with our setting:
\begin{itemize}
	\item $\cX_0^{n-1}=\P^{n-1}$.
	\item For $1\le k\le n-2$, $\cX_k^{n-1}$ is the blowup of $\cX_{k-1}^{n-1}$ along all $\overline{\l}_\alpha$ with $\alpha\subset [n]$ and $|\alpha|=k$.
	\item $\overline{\l}_\alpha$ is isomorphic to $\cX^{k-1}$, thus the codimension of $\overline{\l}_\alpha$ is $n-k$.
\end{itemize}
This gives us 
	\[
	H^*(\cX_k^{n-1}) \cong H^*(\cX_{k-1}^{n-1})\oplus\left(
    \bigoplus_{\stackrel{\alpha\subset[n],}{|\alpha|=k}} \bigoplus_{i=1}^{n-k-1} H^*_{-2i}(\overline{\l}_\alpha)\right)
    \cong H^*(\cX_{k-1}^{n-1})\oplus\left(
    \bigoplus_{\stackrel{\alpha\subset[n],}{|\alpha|=k}} \bigoplus_{i=1}^{n-k-1} H^*_{-2i}(\cX^{k-1})\right).
	\]
The isomorphism in the theorem can be obtained by applying the above isomorphism inductively on $k$. 
\end{proof}

In the following, we compute some examples in low dimensions.

\begin{ex}
	For the base case, $H^*(\cX_0^{n-1}) \cong H^*(\P^{n-1}) \cong \C[\xi]/(\xi^n)$ where $\xi\in H^2(\P^{n-1})$ is the first Chern class of the hyperplane bundle. Next, let us consider $\cX_1^{n-1}$. This is the space obtained in the first step of the blowups. Here, we blowup the points $Z_1,\cdots,Z_n$. The cohonology groups become
	
	\begin{align*}
		H^0(\cX_1^{n-1}) & \cong H^0(\P^{n-1}) \cong \C\\
		H^{2j}(\cX_1^{n-1}) & \cong H^0(\P^{n-1}) \oplus \bigoplus_{i=1}^n H^0(Z_i) \cong \C\oplus\C^{n}, 
		         \text{ $1\le j\le n-2$} \\ 
		H^{2n-2}(\cX_1^{n-1}) & \cong H^{2n-2}(\P^{n-1}) \cong \C. \\
	\end{align*}
In particular, if $n=3$, then $\cX_1^2$ is the blowup of $\P^2$ at the three points $[1:0:0], [0:1:0]$ and $[0:0:1]$. $H^0(\cX_1^2)\cong\C$; $H^1(\cX_1^2)\cong\C\oplus\C^2$ and $H^2(\cX_1^2)\cong\C$.
\end{ex}

\begin{ex}
We investigate the cohomology of $\cX^{n-1}_2$ next. in order to obtain the space, we blowup all the lines $\l_{\{i,j\}}$ where $\{i,j\}\subset [n]$ and $i\ne j$. Notice $\l_{\{i,j\}}$ is isomorphic to its strict transform $\overline{\l}_{\{i,j\}}$ since its dimension is one. We have $H^0(\l_{\{i,j\}})\cong\C$ and $H^2(\l_{\{i,j\}})\cong\C$. Therefore, by Theorem~\ref{thm:cohomology_preperm}, we have 
	\begin{align*}
	H^0(\cX_2^{n-1}) & \cong H^0(\P^{n-1}) \cong \C\\
	H^2(\cX_2^{n-1}) & \cong H^0(\P^{n-1}) \oplus\bigoplus_{\stackrel{i,j \in [n],}{i\ne j}} H^0(\l_{\{i,j\}}) \cong \C \oplus \C^{\binom{n}{2}} \\
	H^{2j}(\cX_2^{n-1}) & \cong H^0(\P^{n-1}) \oplus \bigoplus_{i=1}^n H^0(Z_i) 
	\oplus\bigoplus_{\stackrel{i,j \in [n],}{i\ne j}} \left( H^0(\l_{\{i,j\}}) \oplus H^0(\l_{\{i,j\}})\right) \\
	\mbox{}& \cong \C \oplus\C^{n} \oplus \C^{\binom{n}{2}} \oplus \C^{\binom{n}{2}},\ \ \ 
	\text{ $2\le j\le n-3$}. \\ 
	H^{2n-4}(\cX_2^{n-1}) & \cong H^0(\P^{n-1}) \oplus\bigoplus_{\stackrel{i,j \in [n],}{i\ne j}} H^2(\l_{\{i,j\}}) \cong \C \oplus \C^{\binom{n}{2}} \\
	H^{2n-2}(\cX_2^{n-1}) & \cong H^{2n-2}(\P^{n-1}) \cong \C. \\
\end{align*}

\end{ex}

%
%
%

\subsection{Encoding the cohomology groups and the $\fS_n$ representation}
There is a nice way to encode the one-dimensional components of the cohomology groups of the permutohedral variety by the codes defined by Stembridge~\cite{Ste}. We recall the notations first. For a sequence $\ba = (a_1,\cdots,a_n)$ of nonnegative integers, we will call $n$ the {\em length} of $\ba$. Let $S^+(\ba)=\{a_i|1\le i\le n,a_i>0\}$ denote the set of positive integers in $\ba$. For a positive integer $k$, the sequence $\ba$ is called {\em $k$-admissible} if $S^+(\ba)=\{1,\cdots,k\}$; $\ba$ is {\em $0$-admissible} if $S^+(\ba)=\emptyset$, i.e. $\ba$ consists of all $0$'s. The sequence $\ba$ is called {\em admissible} if it is $k$-admissible for some $k\ge 0$.

Let $m_j(\ba)$ denote the number of occurrences of the integer $j$ in the sequence $\ba$. A {\em marked sequence} is a pair $(\ba, f)$ where $\ba$ is a sequence of nonnegative integers and $f:S^+(\ba)\to \N$ be a map such that $1\le f(j) < m_j(\ba)$ for all $j\in S^+(\ba)$. 
We will adapt the notation in~\cite{Ste} and represent a marked sequence $(\ba, f)$ by putting a hat notation on top of $j$ at the $[f(j)+1]$-st occurrence of $j$. The {\em index} of a marked sequence is defined as $\ind(\ba,f):=\sum_{j\in S^+(\ba)} f(j)$.

A {\em code} is defined to be an admissible marked sequence, i.e. a marked sequence $(\ba,f)$ such that $\ba$ is admissible.
There is one code $((0,\cdots,0),\emptyset)$ consists of all $0$'s and the empty function as the marking function. We define the index of this code to be $0$.

For a code $\ba$, let $\max(\ba)=\max\{a_i|i=1,\cdots,n\}$ be the maximum number in $\ba$, and let $\mu(\ba):=m_{\max(\ba)}(\ba)$ be the number of occurrences of $\max(\ba)$ in $\ba$. For example, if $\ba=1201\hat{2}\hat{1}2$, then $\max(\ba)=2$ and $\mu(\ba)=3$. We further define $\ba'$ to be the sequence obtained after removing all the $\max(\ba)$ from $\ba$, but keep everything else unchanged. For the example $\ba=1201\hat{2}\hat{1}2$ above, one would get $\ba'= 101\hat{1}$. It is possible to have $\ba'$ equals to the empty sequence, in which case $\ba$ consists all $0$'s or all $1$'s. Notice that if $\ba$ is a code, then $\ba'$ is either a code or an empty sequence. 
 
Recall, from Theorem~\ref{thm:cohomology_preperm}, that we have the decomposition of the cohomology of the permutohedral variety, by setting $k=n-2$ in (\ref{eq:cohomology_preperm}), as follows.
\begin{equation}
	\label{eq:cohomology_permutohedral_var}
	H^*(\cX^{n-1}) \cong H^*(\P^{n-1})\oplus\left(
	\bigoplus_{j=1}^{n-2} \bigoplus_{\stackrel{\alpha\subset[n],}{|\alpha|=j}} \bigoplus_{i=1}^{n-j-1} H^*_{-2i}(\overline{\l}_{\alpha})\right).
\end{equation} 
 The main result of this section is the following.
 
\begin{prop}
	\label{prop:encoding_cohomology}
Let $n\ge 2$ be an integer. There is a natural one-to-one correspondence between codes of length $n$ and one-dimensional components of $H^*(\cX^{n-1})$.
\end{prop}

\begin{proof}	
We will construct the correspondence inductively. The base case is a code $\ba$ with $\mu(\ba)=n$. That means, $\ba$ consists of either all $0$'s or all $1$'s. In this case, suppose $\ind(\ba,f)=i$, we assign $(\ba,f)$ to the $2i$-th cohomology of $H^*(\P^{n-1})$ in the above decomposition.	

Next, suppose that $\mu(\ba)<n$. Then we must have $\max(\ba)>0$ and thus $2\le\mu(\ba)\le n-1$. We first set the integers $i,j$ and the subset $\alpha\subset[n]$ that are used as indices in the above decomposition. 
\begin{itemize}
	\item $j=n-\mu(\ba)$,
	\item $\alpha=\left\{i\, |\, a_i<\max(\ba)\right\}$,
	\item $i=f(\max(\ba))$.
\end{itemize}
Notice that, since $\max(\ba)>0$, $i=f(\max(\ba))$ is well-defined; moreover, we have $1\le i\le \mu(\ba)-1 = n-j-1$ by the definition of $f$. Concretely, what these indices record are the following:
\begin{itemize}
	\item $\mu(\ba)$ record the codimension of the subspace which is blown up in the process of constructing $\cX^{n-1}$.
	\item $\overline{\l}_{\alpha}$ is the subspace getting blown up in the process of constructing $\cX^{n-1}$. Equivalently, the set
	\[
	[n]\setminus\alpha=\left\{i\, |\, a_i=\max(\ba)\right\}
	\] 
	encodes the coordinates of the subspace that are set to be $0$. For example, suppose $n=6$ and $\ba=12\hat{1}01\hat{2}$, then $\mu(\ba)=2$, $\alpha=\{1,3,4,5\}$, and the subspace getting blown up is defined by the equations $z_2=z_6=0$ in $\P^5$ (then taking the strict transform to the proper space).  
	\item $i=f(\max(\ba))$ denotes the shifting of degrees when we map $H^*(\overline{\l}_{\alpha})$ into $H^*(\cX^{n-1})$.
\end{itemize}
Next, $\ba'$ is a code of length $j=n-\mu(\ba)$. Therefore, by induction hypothesis, it corresponds to a unique component in $H^*_{-2i}(\overline{\l}_{\alpha})\cong H^*_{-2i}(\cX^{j-1})$. 

Finally, since one can recover $\ba$ uniquely from $\ba'$, $\alpha$, and $i$, this correspondence is one-to-one. One can also observe that the ranges of $j$ and $i$ we defined from codes of length $n$ are the same as the range for the indices in the decomposition (\ref{eq:cohomology_permutohedral_var}). Therefore the correspondence is surjective.   
\end{proof}

Recall that, for integers $n\ge 2$ and $0\le k\le n-2$, the prepermutohedral variety $\cX^{n-1}_k$ is obtained from $\cX^{n-1}_0=\P^{n-1}$ by blowing up points, lines, ..., all the way to the coordinate subspaces of dimension $k-1$. Since $\mu(\ba)$ records the codimension of the subspace getting blown up, one can conclude the following coding for the cohomologies of $\cX^{n-1}_k$ as well.  

\begin{cor}
	For integers $n\ge 2$ and $0\le k\le n-2$, there is a natural one-to-one correspondence between one-dimensional components of $H^*(\cX^{n-1}_k)$ and the set
	of codes $\ba$ of length $n$ such that $\mu(\ba)\ge n-k$.
\end{cor}

\begin{ex}
		For $n=4$, we can read the cohomologies of $\cX^3_k$, $k=0,1,2$, and the corresponding codes, from the following table:	
	\begin{center}	{
			\[
			\setstretch{1.5}
			\begin{array}{|c|cc|cc|cc|}
				\hline
				\mbox{} & \multicolumn{2}{|c|}{\cX^3_0} & \multicolumn{2}{|c|}{\cX^3_1} & \multicolumn{2}{|c|}{\cX^3_2} \\ \hline
				& &\ \text{codes}
				&&\ \text{codes}   &&\ \text{codes} \\ \hline
				H^0(\cX^3_\bullet) \cong  & H^0(\P^3)& 0000
				&   & &&\\ \hline
				H^2(\cX^3_\bullet) \cong  & H^2(\P^3)& 1\hat{1}11
				& \oplus_{j=1}^4 H^0(\l_j) & 01\hat{1}1
				& \oplus_{\stackrel{i,j \in [n],}{i\ne j}} H^0(\l_{\{i,j\}}) & 001\hat{1} \\ \hline
				H^4(\cX^3_\bullet) \cong  & H^4(\P^3) & 11\hat{1}1
				& \oplus_{j=1}^4 H^0(\l_j) & 011\hat{1}
				& \oplus_{\stackrel{i,j \in [n],}{i\ne j}} H^2(\l_{\{i,j\}}) & 1\hat{1}2\hat{2}
				\\ \hline
				H^6(\cX^3_\bullet) \cong  & H^6(\P^3)& 111\hat{1}
				&  & &&\\ \hline
			\end{array}
			\]
		}
	\end{center}
	Entries in the table record the direct summands needed to obtain the cohomology of the corresponding spaces. For example, one can read from the table that  
	$H^2(\cX^3_1) \cong H^2(\P^3) \oplus_{j=1}^4 H^0(\l_j)$ and 
	$H^4(\cX^3_2) \cong   H^4(\P^3)
	\oplus_{j=1}^4 H^0(\l_j) 
	\oplus_{\stackrel{i,j \in [n],}{i\ne j}} H^2(\l_{\{i,j\}})$. Notice that $\l_j$ here are points and $\l_{\{i,j\}}$ are lines.
	The codes recorded here are the representatives for all those in the direct sum (i.e. in an $\fS_4$ orbit). For instance, the code $01\hat{1}1$ represents the four codes 
	$01\hat{1}1, 10\hat{1}1, 1\hat{1}01$, and $1\hat{1}10$. We pick the representative to be the one that is increasing in numbers.
\end{ex}

The convenience of using the codes to encode components of cohomology comes in two ways. First, it give a nice compact way to write $H^*(\cX^{n-1}_k)$. Let $\Code(n)$ denote the set of all codes of length $n$, then
\[
H^*(\cX^{n-1}_k) = \bigoplus_{\stackrel{\ba\in\Code(n)}{\mu(\ba)\ge n-k}}  \C_\ba.
\]
Second, the $\fS_n$-representation on $H^*(\cX^{n-1}_k)$ is compatible with the $\fS_n$ action on $\Code(n)$, in the sense that for any $w\in\fS_n$, we have $w\, \C_\ba=\C_{w\ba}$ in the above decomposition of $H^*(\cX^{n-1}_k)$. We explain this assertion in the next paragraph.

For $w\in\fS_n$, the action of $\fS_n$ on $\Code(n)$ is given by $w\cdot(\ba,f) = (w\ba,f)$ where $(w\ba)_i= a_{w^{-1}(i)}$, $1\le i\le n$. The action on $\cX^{n-1}_k$ is induced from the action on $\P^{n-1}$ given by permuting the coordinates, i.e.
$ w\cdot [z_1:\cdots:z_n] = [z_{w(1)}:\cdots : z_{w(n)}] $.
Then the action on the cohomology is given by the pullback of the action on $\cX^{n-1}_k$. For the coordinate hyperplane $(z_i=0)$, its pullbak is $(z_{w^{-1}(i)}=0)$. Recall that, in the encoding, the locations of the maximal numbers corresponds to coordinates to set to be $0$. Moreover, permuting the coordinates also results in permuting the normal directions to the coordinate linear subspaces, and the exceptional divisor corresponds to the projective bundle associated to the normal bundle. Hence we know that $w$ maps $H^*(\overline{\l}_{\alpha})$ to $H^*(\overline{\l}_{w^{-1}\alpha})$. Then, we can conclude $w\, \C_\ba=\C_{w\ba}$ by induction on $n$.

Therefore, the $\fS_n$ representation on $\cX^{n-1}_k$ is the permutation representation induced by the permutation on codes of length $n$ with $\mu(\ba)\ge n-k$. An immediate consequence of this fact is the following.  

\begin{prop}\label{Thm:permutaion_basis}
	There is a permutation basis for the dot representation on the cohomology of the prepermutohedral varieties $\cX_{k}^{n-1}$, $1\le k\le n-2$. 
\end{prop}

\begin{proof}
For each $\fS_n$ orbit of codes under the encoding, pick a representative $\ba$ and pick an element $\xi\in\C_{\ba}$. Setting all other components to be $0$ gives us an element in $H^*(\cX_k)$. Then the $w\xi$'s, as $w$ runs through $\fS_n$, form a linearly independent set. As we pick these elements for all orbits, we obtain a permutation basis.   
\end{proof}

\subsection{The characteristic series for the dot action on prepermutohedral varieties}

We follow the notation in \cite{Procesi} and denote
\[
A_{n-1}(t) = \sum_{j=0}^{n-1} \ch H^{2j} (\cX^{n-1}) t^j,
\]
where $H^{2j} (\cX^{n-1})$ denotes the $\fS_n$-representation, and `$\ch$' is the Frobenius characteristic map. In \cite{Procesi}, Procesi observed the following recursive formula 
\[
A_{n-1}(t) = s_n \sum_{i=0}^{n-1} t^i + \sum_{i=0}^{n-3} s_{n-1-i} A_i(t) \left(\sum_{l=1}^{n-i-2} t^l\right)
\] 
from the iterated blowup structure of $\cX^{n-1}$. Here $s_n$ and $s_{n-1-i}$ are the Schur symmetric functions. The first term in the recursive relation corresponds to the representation on base space $\P^{n-1}$, and for each $i=0,\cdots, n-3$, the term $s_{n-1-i} A_i(t) \sum_{l=1}^{n-i-2} t^l$ corresponds to the representation on the blowup of $i$-dimensional coordinate space.   

For a prepermutohedral variety $\cX^{n-1}_k$, we denote
\[
A_{n-1,k}(t) := \sum_{j=0}^{n-1} \ch H^{2j} (\cX^{n-1}_k) t^j.
\]
Since we also have the iterated blowup structure for $\cX^{n-1}_k$, we have a similar recursive formula for $A_{n-1,k}(t)$, with the second summation only goes from $0$ to $k-1$.
With the more compact notation $[n]_t=\sum_{i=0}^{n-1} t^i$, and with the identity $s_j=h_j$ (the complete homogeneous symmetric functions), the recursive formula can be written as
\begin{equation}
	\label{recursive_A}
A_{n-1,k}(t) = h_n [n]_t + \sum_{i=0}^{k-1} h_{n-1-i} A_i(t) t [n-i-2]_t.
\end{equation}
Notice that for $k=0$, the second term disappears. This recursive formula is sufficient for our purpose, but it might be interesting to derive a close form for $A_{n-1,k}(t)$.

\section{The geometry of certain Hessenberg varieties} In this section, we consider Hessenberg varieties associated with Hessenberg functions of the type $h_k=(2,3,\cdots, k+1, n,\cdots,n)$, for some $k\le n-3$. That is,
\[
h_k(j)=\begin{cases}
	j+1, &  j=1,\cdots,k\\
	n,   &  j=k+1,\cdots,n.
\end{cases}
\]
We denote the corresponding Hessenberg variety $\cY=\Hess(\SS,h_k)$. It is a smooth complex variety of dimension $k+ \frac{1}{2}{(n-k)(n-k-1)}$. There is a morphism from $\cY$ to the Hessenberg type variety $\cX_{k}\cong\Hess^{(k+1)}(\SS,h_+)$ which remembers only the flags up to dimension $k+1$. More precisely, we define 
\[
f:\cY\to \cX_{k}
\]
as follows:
\[
(V_0\subset V_1\subset\cdots\subset V_{k}\subset\cdots\subset V_{n})\longmapsto (V_0\subset V_1\subset\cdots\subset V_{k+1}).
\]
Since $h_k(j)=n$ for $j\ge k+2$, there is no condition imposed from $h_k$ on $V_j$ for $j\ge k+2$. (Notice that
$h_k(k)=k+1$ implies $\SS V_k\subset V_{k+1}$, so there is still condition from $h_k$ for $V_{k+1}$.) Therefore, the fiber of $f$ over a partial flag $(V_0\subset\cdots\subset V_{k+1})$ can be identified with the flag variety $\Flag(\C^n/V_{k+1})\cong\Flag(\C^{n-k-1})$. Hence $\cY$ has a fiber bundle structure over $\cX_{k}$.

Over $\cY$, there is a tautological filtration
\[
\cV_{k+1}\subset\cV_{k+2}\subset\cdots\subset\cV_n \cong \cY\times \C^n
\]
of vector subbundles of the trivial bundle over $\cY$: the fiber of $\cV_j$ over a flag
$(V_0\subset\cdots\subset V_{n})$ is the vector space $V_j$. We then have the line bundles $\cL_j:=\cV_j/\cV_{j-1}$, $k+2\le j\le n$. Set $x_j=-c_1(\cL_j)$ to be the negative of the first Chern class of the line bundle $\cL_j$, then we have the following description of the cohomology ring $H^*(\cY)$.

\begin{prop}
	\label{prop:cohomology_Y}
The cohomology ring $H^*(\cY)=H^*(\cY;\C)$ is generated over $H^*(\cX_{k})$ by the classes $x_{k+2},\cdots,x_n$, subject to the relations $e_i(x_{k+2},\cdots,x_n)=0$ for $1\le i\le n-k-1$. That is,
\[
H^*(\cY)\cong H^*(\cX_{k})[X_{k+2},\cdots,X_n]/
(e_1(X_{k+2},\cdots,X_n),\cdots,e_{n-k-1}(X_{k+2},\cdots,X_n)),
\]
and the classes $x_{k+2},\cdots, x_n$ can be identified with the images of $X_{k+2},\cdots,X_n$ under the quotient. In addition, the classes $x_{k+2}^{i_{k+2}}\cdots x_n^{i_n}$, with exponents $0\le i_{j}\le n-j$, form a basis for $H^*(\cY)$ over $H^*(\cX_{k})$. 
\end{prop}

\begin{proof}
	The proof mimics the proof for the cohomology of the flag variety \cite[Proposition~10.2.3]{Fulton} and is based on basic facts about projective bundles. For more details, see \cite[p.~606]{GrHa} or \cite[Appendix B.4]{Fulton2}. For a vector bundle $\cV$ over a variety $\cX$, let $\rho:\P(\cV)\to \cX$ denote the corresponding projective bundle. There is a tautological bundle $\cL\subset\rho^*(\cV)$. Set $\xi=-c_1(\cL)$, then
	\begin{equation}\label{equation:cohomology_of_proj_bundles}
	H^*(\P(\cV))\cong H^*(\cX)[\xi]/(\xi^r+a_1\xi^{r-1}+\cdots+a_r),
	\end{equation}
	where $a_i=c_i(\cV)\in H^{2i}(\cX)$. 
	
	In this proof we suppress all the notions of pullbacks of bundles. One can construct $\cY$ from $\cX_{k}$ as a sequence of projective bundles. First, over $\cX_{k}$ there is a bundle $\cU \to \cX_{k}$ of rank $n-k-1$ whose fiber over the flag $(V_0\subset\cdots\subset V_{k+1})$ is the vector space $\C^n / V_{k+1}$. The projective bundle $\P(\cU)$ is the first bundle in the sequence. It gives the direction of the extra dimension of $V_{k+2}$ over the flag $(V_0\subset\cdots\subset V_{k+1})$. The tautological bundle $\cU_1$ of $\P(\cU)$ pulls back to the line bundle $\cL_{k+2}$ on $\cY$. 
	
	Next, over $\P(\cU)$, we have the bundle $\cU/\cU_{1}$ of rank $n-k-2$, and we construct the second projective bundle $\P(\cU/\cU_1)\to\P(\cU)$. The tautological bundle of $\P(\cU/\cU_{1})$ is of the form $\cU_2/\cU_1$ for some vector bundle $\U_2$ of rank $2$ and $\cU_1\subset\cU_2\subset\cU$ as bundles over $\P(\cU/\cU_{1})$. Moreover, the tautological bundle of $\P(\cU/\cU_{1})$ pulls back to $\cL_{k+3}$ on $\cY$. One can continue this process and construct $\P(\cU/\cU_{2})$ as a projective bundle over $\P(\cU/\cU_{1})$, with tautological bundle of the form $\cU_3/\cU_2$, and so on. At the end, one arrive at the space $\P(\cU/\cU_{n-k-1})$ which is isomorphic to the Hessenberg variety $\cY$.
	Therefore, by the formula \ref{equation:cohomology_of_proj_bundles} and the fact that the tautological bundle of $\P(\cU/\cU_j)$ pulls back to the line bundle $\cL_{k+j+1}$ on $\cY$, we obtain the conclusion that $H^*(\cY)$ is generated by $x_{k+2},\cdots,x_n$ over $H^*(\cX_k)$.
	
	The rest of the proof is also very similar to the proof of \cite[Proposition~10.2.3]{Fulton}. The reason for $e_i(x_{k+2},\cdots,x_n)=0$ is because it is the $i$-th Chern class of the trivial bundle $\cV_n$ defined above; and the isomorphism in the proposition is based on the same algebraic fact stated on \cite[p.163]{Fulton}. 
\end{proof}

The dot action on $H^*(\cX_{k})$ was described in Section~\ref{sec:dot_on_preperm}. Motivated by the fact that the dot action acts trivially on the usual cohomology of the flag variety, it is natural to conjecture that the action on the basis classes $x_{k+2}^{i_{k+2}}\cdots x_n^{i_n}$ are trivial. While we can not prove this result for the classes, we can prove the result on the isomorphism level with an argument using characteristics series of the representation. The author would like to thank John Shareshian for bringing the article \cite{HNY} to his attention.

\begin{prop}
\label{prop:dot_on_Y}
	The dot representation on $\cY$ is isomorphic to the representation on 
\[
H^*(\cX_{k})[X_{k+2},\cdots,X_n]/
(e_1(X_{k+2},\cdots,X_n),\cdots,e_{n-k}(X_{k+2},\cdots,X_n))
\]
which acts on  $H^*(\cX_{k})$ as described in Section \ref{sec:dot_on_preperm}, and acts trivially on the $H^*(\cX_{k})$-basis $x_{k+2}^{i_{k+2}}\cdots x_n^{i_n}$, $0\le i_{j}\le n-j$, where the $x_i$'s are the images of the $X_i$'s under the quotient map.
\end{prop}

\begin{proof}
We have the recursive relation for the characteristic series of $\cX_k$ (\ref{recursive_A}), with some terms rearranged:
\begin{equation}
\label{eqn:recursive_A2}
A_{n-1,k}(t) = h_n [n]_t + \sum_{i=0}^{k-1} t [n-i-2]_t 
 A_i(t) h_{n-1-i}.
\end{equation}
Suppose that $S_{n}$ acts trivially on $X_{k+2}^{i_{k+2}}\cdots X_n^{i_n}$, $0\le i_{j}\le n-j$, then on the characteristic series, what happen would be a degree shifting for each basis element, i.e. multiplying by a power of $t$. For $X_{k+2}^{i_{k+2}}$, $0\le i_{k+2}\le n-k-2$, it corresponds to multiplying $1+t+\cdots+t^{n-k-2} = [n-k-1]_t$. Finally, the effect of taking the basis $X_{k+1}^{i_{k+2}}\cdots X_n^{i_n}$, $0\le i_{j}\le n-j$ in to consideration on the characteristic series, assuming $S_n$ acts trivially on them, would be multiplying by 
\[
[n-k-1]_t [n-k-2]_t \cdots [1]_t = [n-k-1]_t !
\]

The incomparability graph of the Hessenberg function $h_k$ is the lollipop graph $L_{n-k,k}$. For the chromatic quasisymmetric function $X_{Ln-k,k}(\bx,t)$, there is a recursive formula \cite[Proposition 4.4]{HNY}
\begin{equation}
	X_{L_{n-k,k}}(\bx,t)=[n-k-1]_t !\left(
	[n]_t e_n + \sum_{i=0}^{k-1} 
	t [n-k+i-1]_t X_{P_{k-i}}(\bx,t) e_{n-k+i}\right)
\end{equation}
Setting $i' = k-1-i$ in the sum, we have
\begin{equation}
\label{eqn:recursive_X_L}
	X_{L_{n-k,k}}(\bx,t)=[n-k-1]_t !\left(
	[n]_t e_n + 
	\sum_{i'=0}^{k-1} 
	t [n-i'-2]_t X_{P_{i'+1}}(\bx,t) e_{n-1-i'}\right)
\end{equation}

Notice that the incomparability graph of the Hessenberg function $h_+$ for $\cX^{i}$ is $P_{i+1}$, and it is known that $\omega X_{P_{i+1}}(\bx,t)=A_i(t)$, where $\omega$ is the involution on the ring of symmetric functions. By definition, $\omega e_i=h_i$, therefore, comparing the right sides of (\ref{eqn:recursive_A2}) and (\ref{eqn:recursive_X_L}), we obtain the following.

\[
\omega X_{L_{n-k,k}}(\bx,t)=[n-k-1]_t! A_{n-1,k}(t)
\]

Finally, by the theorem that was originally conjectured by Shareshian and Wachs~\cite[Conjecture 1.4]{SW} and proved independently by Brosnan and Chow~\cite{BC}, and by Guay-Paquet~\cite{GP}, we have the identity.

\[
\sum_{j=0}^{\dim(\cY)} \ch H^{2j} (\cY) t^j =
\omega X_{L_{n-k,k}}(\bx,t)
\]

Therefore we can conclude that $\sum_{j=0}^{\dim(\cY)} \ch H^{2j} (\cY) t^j = [n-k-1]_t! A_{n-1,k}(t)$. Since the representation is determined by the characteristic series up to isomorphism, this concludes the proof.
\end{proof}

We remark that the proposition also implies the dot representation is a permutation representation on $H^* (\cY)$.


\end{document}